\documentclass[final]{siamltex}

\usepackage{epsfig}
\usepackage{amsmath,amssymb}
\usepackage {graphicx}
\usepackage{amsmath,amssymb}
\usepackage{epsfig}
\usepackage{amsmath,amssymb}
\usepackage{graphicx}
\usepackage{epstopdf}
\usepackage{amsmath,amssymb}

\newtheorem{Cor}{Corollary}[section]
\newtheorem{prop}{Proposition}[section]

\newtheorem{rem}{Remark}[section]
\newtheorem{algorithm}{Algorithm}

\title{A New Method for Computing $\varphi$-functions and Their Condition Numbers of Large Sparse Matrices}

\author{Gang Wu\thanks{Corresponding author (G. Wu). Department of Mathematics,
China University of Mining and Technology, Xuzhou, 221116, P.R. China.
E-mail: {\tt gangwu76@126.com} and {\tt gangwu@cumt.edu.cn}. This author is
supported by the National Science Foundation of China under grant 11371176, the Natural Science Foundation of Jiangsu
Province under grant BK20131126, and the Talent Introduction Program of China
University of Mining and Technology.}
        \and Lu Zhang\thanks{School of Mathematics and Physical Sciences Technology, Xuzhou Institute of Technology, Xuzhou, 221111, Jiangsu, P.R. China.
E-mail: {\tt yulu7517@126.com}.}
}

\begin{document}

\maketitle

\begin{abstract}
We propose a new method for computing the $\varphi$-functions
of large sparse matrices with low rank or fast decaying singular values.
The key is to reduce the computation of $\varphi_{\ell}$-functions of a large matrix
to $\varphi_{\ell+1}$-functions of some $r$-by-$r$ matrices, where $r$ is the numerical rank of the large matrix in question.
Some error analysis on the new method is given.
Furthermore, we propose two novel strategies for estimating 2-norm condition numbers of the $\varphi$-functions.
Numerical experiments illustrate the numerical behavior of the new algorithms and show the effectiveness of our theoretical results.
\end{abstract}



\pagestyle{myheadings} \thispagestyle{plain} \markboth{G. WU AND L. ZHANG}{\sc New Method for Computing $\varphi$-Functions and Their Condition Numbers}

\section{Introduction}

\setcounter{equation}{0}

In recent years, a great deal of attention has been focused on
the efficient and accurate
evaluation of matrix functions closely related to the $\varphi$-functions
\cite{Ber2,Higham,MCH,Hochbruck-Ostermann-2010,YY,ML,NW2,ST,WMWright,Wu1}.
For instance, exponential integrators make use of the matrix exponential and related matrix functions
within the formulations of the numerical methods, and
the evaluation of matrix functions is crucial for accuracy, stability, and efficiency
of exponential integrators \cite{Hochbruck-Ostermann-2010}.
The $\varphi$-functions
are defined for scalar arguments by the integral representation as follows
\begin{equation}\label{equ1}
\varphi_{0}(z)={\rm exp}(z)\quad{\rm and}\quad\varphi_{\ell}(z)=\frac{1}{(\ell-1)!}{\int_0^1 {\rm exp}\big((1-\theta)z\big)\theta^{\ell-1}d\theta},\quad \ell=1,2,\ldots
\end{equation}
Moreover, the $\varphi$-functions satisfy the following recurrence relations
\begin{equation}\label{eqn11}
\varphi_{\ell}(z)=z\varphi_{\ell+1}(z)+\frac{1}{\ell!},\quad \ell=0,1,2,\ldots
\end{equation}
This definition can be extended to matrices instead of scalars by using any of the
available definitions of matrix functions \cite{Higham,Hochbruck-Ostermann-2010}.

In a wide range of applications, such as the matrix exponential discriminant analysis method for data dimensionality reduction \cite{Ah,DB,WCP,Wang,YPan,ZF}, and the complex network analysis method based on matrix function \cite{Benzi,EH,EDH,EJA}, it is required to compute the matrix exponential with respect to large scale and low-rank matrix.
In this paper, we are interested in computing several $\varphi$-functions {\it consecutively},
with respect to a large scale matrix $A$ with low rank or with fast decaying singular values. Let $\sigma_j$ be the $j$-th largest singular value of $A$, by ``fast decaying singular values", we mean $\sigma_j=\mathcal{O}(\rho^{-j}),\rho>1$ or $\sigma_j=\mathcal{O}(j^{-\alpha}),\alpha>1$ \cite{Hof}.
Such matrices appear frequently in diverse application areas
such as data dimensionality reduction \cite{DHS}, complex network analysis \cite{EDH}, discretizing ill-posed operator equations that model many inverse problems \cite{Han},
randomized algorithms for matrix approximations \cite{Random,Mon}, finite elements discretization \cite{Testmatrix}, and so on.

In spite of the high demand for efficient methods to solve the matrix $\varphi$-functions in various fields of computational sciences,
there is no easy way to solving this type of problem.
Indeed, when $A$ is large, both the computational cost and the storage requirements are prohibitive, moreover, $\varphi_{\ell}(A)$ can be dense even if $A$ is sparse \cite{Higham}.
Some available methods are only suitable for medium-sized matrices.
For instance, a MATLAB toolbox called EXPINT is provided by Berland, Skaflestad and Wright \cite{Ber2} for this problem. Kassam and Trefethen \cite{Tre} propose to approximate $\varphi$-functions with a contour integral,
which worked well as long as the contour of integration is suitably chosen. However, the
contour is in general problem-dependent and difficult to determine in advance. Another way is to reduce the computation of $\varphi_{\ell}(A)$ to that of matrix exponential with larger size \cite{HJH,YY,Sidje}, which is unfeasible as the matrix in question is large. The third way is based on a modification of the scaling
and squaring technique \cite{WMWright}, the most commonly used approach for computing the matrix exponential \cite{NJ}.
The
most well-known methods for computing $\varphi$-functions for large spares matrices
are the Krylov subspace methods \cite{NW2,Wu1}. However, the Krylov subspace methods are applicable to the computation of $\varphi$-functions on given (block) vectors, while the main aim of this paper is to compute $\varphi$-functions of large sparse matrices.

In practical calculations, it is important to know how accurate the computed solution is and how small perturbations in input data can effect outputs \cite{NJH}.
Therefore, it is crucial to give error analysis and understand the sensitivity of matrix function to perturbation in the data. Sensitivity is measured by condition numbers. For matrix function, condition number can be expressed in terms of the norm of the Fr\'{e}chet derivative,
and it is often measured by using the 1-norm \cite{Higham,Toolbox}. In practice, however, the 2-norm is a more widely used norm than the 1-norm, and the former is preferable for both theoretical analysis and computational purposes.
In this work, we consider how to evaluate the Fr\'{e}chet 2-norm condition numbers of $\varphi$-functions effectively.

Given a large scale matrix $A$ with low rank or with fast decaying singular values, we propose a new method for evaluating several $\varphi$-functions and their absolute and relative condition numbers {\it consecutively}.
Our new method is based on the sparse column-row approximation of large sparse matrices \cite{Stewart,GWStewart,ErrStewart}. An advantage is that there is no need to explicitly form and store the $\varphi$-functions or the Fr\'{e}chet derivatives with respect to $A$.
The overhead is only to compute $\varphi$-functions of some $r$-by-$r$ matrices, and to store two $n$-by-$r$ sparse matrices, where $r$ is the (numerical) rank of $A$.
This paper is organized as follows. In Section 2, we present the main algorithm, and give some error analysis on the proposed method.
In Section 3, we propose two novel strategies for estimating the absolute and relative 2-norm condition numbers of $\varphi$-functions. In Section 4, numerical experiments are given to illustrate the efficiency of our new strategies. Some concluding remarks are given in Section 5.

Some notations used are listed as follows. Throughout this paper, we denote by $\widetilde{A}=XTY^T$ a {sparse column-row approximation} to $A$.
Let $\|\cdot\|_2,\|\cdot\|_1$ be the 2-norm and the 1-norm of a vector or matrix, and $\|\cdot\|_F$ be the Frobenius norm of a matrix.
We denote by $\otimes$ the Kronecker product, and by ${\rm vec}(\cdot)$  the ``vec operator" that stacks the columns of a matrix into a long vector. Let $I$ and $O$ be the identity matrix and zero matrix, respectively, whose order is clear from context.
We only focus on real matrices in this paper. Indeed, all the results can be extended to complex matrices in a similar way.


\section{A new method for $\varphi$-functions of large sparse matrices}
\setcounter{equation}{0}

In this section, we will present a new method for $\varphi$-functions of large sparse matrices with low rank or with fast decaying singular values, and give some error analysis on the proposed method.
Given an $n\times n$ large sparse matrix $A$, we first find a reduced-rank approximation
$XTY^T$ to $A$, where both $X$ and $Y$ are full column rank and $T$ is nonsingular. This type of problem arises in a number of applications such as information retrieval, computational biology and complex network analysis \cite{Ber1,Ber22,Ber3,Stewart,Jiang,Stuart,ZY}.
A widely used reduced-rank approximation is the truncated singular value
decomposition (TSVD) \cite{GV}, which is known to be optimal in the sense that the Frobenius
norm $\|A-XTY^T\|_F$ is minimized. Unfortunately, this method computes the full decomposition and is
not suitable for very large matrices.
An alternative is the randomized singular value
decomposition algorithm \cite{Random,Mon}, which generally gives
results comparable to TSVD. However, for a large and sparse matrix $A$, the situation is not so simple: the storage requirements
and operation counts will become proportional to the number of nonzero
elements in $A$. Since the resulting factors $X, T$, and $Y$ are generally not sparse, one may suffer from heavily
computational cost.

In \cite{GWStewart}, Stewart introduced a quasi-Gram-Schmidt
algorithm that produces a sparse QR factorization to $A$. Based on the quasi-Gram-Schmidt
algorithm, a sparse column-row approximation algorithm was proposed.
This algorithm first applies the quasi-Gram-Schmidt algorithm to the columns of
$A$ to get a representative set of columns $X$ of $A$ and an upper triangular matrix
$R$. Let the error in the corresponding reduced-rank decomposition
be $\epsilon_{\rm col}$. It then applies the same algorithm to $A^T$ to get a representative
set $Y^T$ of rows and another upper triangular matrix $S$. Let the error be $\epsilon_{\rm row}$.
Then the sparse column-row approximation method seeks a matrix $T$ such that $\|A-XTY^T\|_F^2=\min$,
and the minimizer turns out to be \cite{Stewart,GWStewart}
$$
T=R^{-1}R^{-T}(X^TAY)S^{-1}S^{-T},
$$
moreover, we have \cite{GWStewart}
\begin{equation}\label{eqn2.2}
\|A-XTY^T\|_F^2\leq \epsilon_{\rm col}^2+\epsilon_{\rm row}^2.
\end{equation}
The matrix $XTY^T$ is called a {sparse column-row approximation} (SCR) to $A$, where $X,Y\in\mathbb{R}^{n\times r}$ are sparse and full column rank, $T\in\mathbb{R}^{r\times r}$ is nonsingular, and $r$ is the (numerical) rank of $A$.
In this approximation, $X$ consists of a selection of the columns of $A$, and $Y$ consists of a
selection of the rows of $A$, so that when $A$ is sparse so are both $X$ and $Y$. An error analysis of the quasi-Gram-Schmidt algorithm is given in \cite{ErrStewart}.
One is recommended to see \cite{Stewart,GWStewart} for more details on this algorithm.

Given {\it any} rank-revealing decomposition of $A$, the following theorem shows that the computation of $\varphi_{\ell}(A)$ can be reduced to that of $\varphi_{\ell+1}$ function of an $r\times r$ matrix, where $r$ is the (numerical) rank of $A$.

\begin{theorem}\label{Thm4.1}
Let $XTY^T\in\mathbb{R}^{n\times n}$ be a rank-revealing decomposition of $A$, where $X,Y\in\mathbb{R}^{n\times r}$ and $T\in\mathbb{R}^{r\times r}$. Denote $Z=T(Y^TX)\in\mathbb{R}^{r\times r}$, then
\begin{equation}\label{4.1}
\varphi_{\ell}(XTY^T)=\frac{1}{\ell!}I+X\big[\varphi_{\ell+1}(Z)T\big]Y^T,\quad \ell=0,1,\ldots
\end{equation}
\end{theorem}
\begin{proof}
It follows from the definition of $\varphi$-functions that
\begin{eqnarray*}
\varphi_{\ell}(XTY^T)&=&\sum\limits_{k=\ell}^{\infty}\frac{(XTY^T)^{k-\ell}}{k!}\\
&=&\frac{1}{\ell!}I+\sum\limits_{k=\ell+1}^{\infty}\frac{X(TY^TX)^{k-\ell-1}TY^T}{k!}\\
&=&\frac{1}{\ell!}I+X\bigg(\sum\limits_{k=\ell+1}^{\infty}\frac{(TY^TX)^{k-(\ell+1)}}{k!}\bigg)TY^T\\
&=&\frac{1}{\ell!}I+X\big[\varphi_{\ell+1}(TY^TX)T\big]Y^T.
\end{eqnarray*}
\end{proof}

Let $\widetilde{A}=XTY^T$ be a sparse column-row approximation to $A$, then we make use of $\varphi_{\ell}(\widetilde{A})$ as an approximation to $\varphi_{\ell}(A)$.
The following algorithm can be used to compute several $\varphi$-functions of large sparse matrices with low rank or fast decaying singular values
{\it consecutively}.

\begin{algorithm}\label{Alg.1} {\bf An algorithm for computing $\varphi$-functions of large sparse matrices with low rank or fast decaying singular values}\\
1.~Compute a reduced-rank approximation
$XTY^T$ to $A$ by using, say, the sparse column-row approximation {\rm(}SCR{\rm)} algorithm;\\
2.~Compute $\varphi$-functions of small-sized matrices: $\varphi_{\ell+1}(TY^TX),~\ell=0,1,\ldots,p$;\\
3.~Store $X,Y,T$ and $\varphi_{\ell+1}(TY^TX)$ for $\varphi_{\ell}(\widetilde{A}),~\ell=0,1,\ldots,p$. If desired, form $\varphi_{\ell}(\widetilde{A})$ in terms of \eqref{4.1} and use them as approximations to $\varphi_{\ell}(A)~\ell=0,1,\ldots,p$.
\end{algorithm}
\begin{rem}
Obviously, an advantage of the proposed method is its simplicity.
In conventional methods, one has to pay $\mathcal{O}\big((p+1) n^3\big)$ flops for the computation of $\varphi_{\ell}(A)$ \cite{Ber2,Bey,Higham,WMWright}. Given a sparse reduced-rank approximation to $A$, Theorem \ref{Thm4.1} reduces the computation of $\varphi_{\ell}(A)$ to that of $\varphi_{\ell+1}$ functions with respect to the $r$-by-$r$ matrix $Z=T(Y^TX)$, in $\mathcal{O}\big((p+1) r^3\big)$ flops.
For storage, it only needs to store two $n$-by-$r$ sparse matrices $X,Y$, and some small matrices of size $r$-by-$r$, rather than the $n$-by-$n$ possibly dense matrices $\varphi_{\ell}(A),~\ell=0,1,\ldots,p$.
Thus, the new method can compute $\varphi_{\ell}(A),\ell=0,1,\ldots,p$, consecutively and reduce the computational complexities significantly as $r\ll n$.
\end{rem}

In practice, most data are inexact or uncertain. Indeed, even if the data were exact, the computations will subject to rounding errors. So it is important to give error analysis and understand the sensitivity of matrix function to perturbation in the data. Sensitivity is measured by condition numbers. For matrix function, condition number can be expressed in terms of the norm of the Fr\'{e}chet derivative.
The Fr\'{e}chet derivative $L_f(A,E)$ of a matrix function $f: \mathbb{R}^{n\times n}\rightarrow\mathbb{R}^{n\times n}$ at a point $A\in\mathbb{R}^{n\times n}$ is a linear mapping such that for all $E\in\mathbb{R}^{n\times n}$ \cite[pp.56]{Higham}
\begin{equation}\label{eqn317}
f(A+E)-f(A)-L_f(A,E)=o(\|E\|).
\end{equation}
The absolute and relative condition numbers of a matrix function $f(A)$ are defined as \cite{Higham,Rice}
\begin{equation}\label{Eqn4.9}
{\rm cond}_{\rm abs}(f,A)=\|L_f(A)\|=\max_{E\neq O}\frac{\|L_f(A,E)\|}{\|E\|},
\end{equation}
and
\begin{equation}\label{eqn330}
{\rm cond}_{\rm rel}(f,A)=\frac{\|L_f(A)\|\|A\|}{\|f(A)\|},
\end{equation}
respectively.
Theoretically, the Fr\'{e}chet derivative can be obtained from applying any existing methods for computing the matrix function of a $2n\times 2n$ matrix \cite{Higham}
\begin{equation}\label{eqn38}
f\left(\left[\begin{array}{cc}
A & E \\ O & A
\\\end{array} \right]\right)
=\left[\begin{array}{cc}
f(A) & L_f(A,E) \\ O & f(A)
\\\end{array} \right],
\end{equation}
see, \cite[Algorithm~3.17]{Higham}. However, it requires $\mathcal{O}(n^5)$ flops, assuming that the computation of $L_f(A,E)$ takes $\mathcal{O}(n^3)$ flops \cite{Higham}, which are prohibitively expensive for large matrices.

Let $\widetilde{A}=XTY^T$ be a sparse column-row approximation to $A$, and let $E=A-\widetilde{A}$, then we see from \eqref{eqn2.2} that
$\|E\|_F\leq\sqrt{\epsilon_{\rm col}^2+\epsilon_{\rm row}^2}$. Thus, it is interesting to combine existing error analysis for sparse column-row approximation with the theory of matrix functions to obtain error bounds for the proposed method. We first present the following theorem for Fr\'{e}chet derivatives of $\varphi$-functions.
\begin{theorem} For any matrix $E\in\mathbb{R}^{n\times n}$, we have
\begin{equation}\label{equ-phi}
\varphi_{\ell}(A+E)-\varphi_{\ell}(A)=\int_{0}^{1}\exp\big((1-s)A\big)s^{\ell}E\varphi_{\ell}\big(s(A+E)\big)ds,\quad \ell=0,1,2,\ldots
\end{equation}
\end{theorem}
\begin{proof}
For the matrix exponential, it follows from \cite[pp.238]{Higham} that
\begin{equation}\label{eqn39}
\exp(A+E)=\exp(A)+\int_0^1\exp\big((1-s)A\big)E\exp\big(s(A+E)\big)ds,
\end{equation}
so \eqref{equ-phi} holds for $\ell=0$. When $\ell \geq 1$, let
$
B=\left[\begin{array}{cc}
A &I \\ O & O
\\\end{array} \right]\in\mathbb{R}^{2n\times 2n},
$
then we get \cite{Higham,ST}
$$
 \varphi_{\ell}(B)=\left[\begin{array}{cc}
\varphi_{\ell}(A) &\varphi_{\ell+1}(A) \\ O & \frac{1}{\ell!}I
\\\end{array} \right].
$$
Denote $\widetilde{E}=\left[\begin{array}{cc}
E &O \\ O & O
\\\end{array} \right]\in\mathbb{R}^{2n\times 2n}$, it is seen that
\begin{equation}\label{equ-1}
\varphi_{\ell}(B+\widetilde{E})-\varphi_{\ell}(B)=\left[\begin{array}{cc}
\varphi_{\ell}(A+E)-\varphi_{\ell}(A) &\varphi_{\ell+1}(A+E)-\varphi_{\ell+1}(A) \\ O & O
\\\end{array} \right].
\end{equation}

By induction, we assume that
\begin{equation}\label{equ-phiB}
\varphi_{\ell}(B+\widetilde{E})-\varphi_{\ell}(B)=\int_{0}^{1}\exp\big((1-s)B\big)s^{\ell}\widetilde{E}\varphi_{\ell}\big(s(B+\widetilde{E})\big)ds.
\end{equation}
From
$$
(1-s)B\exp\big((1-s)B\big)=\exp\big((1-s)B\big)(1-s)B,
$$
and
$$
s(B+\widetilde{E})\varphi_{\ell}\big(s(B+\widetilde{E})\big)=\varphi_{\ell}\big(s(B+\widetilde{E})\big)s(B+\widetilde{E}),
$$
we obtain
$$
\exp\big((1-s)B\big)=\left[\begin{array}{cc}
\exp\big((1-s)A\big) &(1-s)\varphi_{1}\big((1-s)A\big) \\ O & I
\\\end{array} \right],
$$
and
$$\varphi_{\ell}\big(s(B+\widetilde{E})\big)=\left[\begin{array}{cc}
\varphi_{\ell}\big(s(A+E)\big) &s\varphi_{\ell+1}\big(s(A+E)\big) \\ O & \frac{1}{\ell!}I
\\\end{array} \right].$$
Thus, 
\begin{eqnarray*}
&& \exp\big((1-s)B\big)s^{\ell}\widetilde{E}\varphi_{\ell}\big(s(B+\widetilde{E})\big)\\
&=&\left[\begin{array}{cc}
\exp\big((1-s)A\big)s^{\ell}E\varphi_{\ell}\big(s(A+E)\big) &\exp\big((1-s)A\big)s^{\ell+1}E\varphi_{\ell+1}\big(s(A+E)\big) \\ O & O
\\\end{array} \right].
\end{eqnarray*}
Furthermore, (\ref{equ-phiB}) can be rewritten as
\begin{eqnarray}\label{equ-phiB-1}
&&\varphi_{\ell}(B+\widetilde{E})-\varphi_{\ell}(B)\\
&=&\left[\begin{array}{cc}
\int_{0}^{1}\exp\big((1-s)A\big)s^{\ell}E\varphi_{\ell}\big(s(A+E)\big)ds &\int_{0}^{1}\exp\big((1-s)A\big)s^{\ell+1}E\varphi_{\ell+1}\big(s(A+E)\big)ds \\ O & O
\\\end{array} \right].\nonumber
\end{eqnarray}
From (\ref{equ-1}) and (\ref{equ-phiB-1}), we have
$$
\varphi_{\ell+1}(A+E)-\varphi_{\ell+1}(A)=\int_{0}^{1}\exp\big((1-s)A\big)s^{\ell+1}E\varphi_{\ell+1}\big(s(A+E)\big)ds,
$$
which completes the proof.
\end{proof}

Using (\ref{equ-phi}) to substitute for $\varphi_{\ell}\big(s(A+E)\big)$ inside the integral, we get
\begin{equation}
\varphi_{\ell}(A+E)=\varphi_{\ell}(A)+\int_{0}^{1}\exp\big((1-s)A\big)s^{\ell}E\varphi_{\ell}(sA)ds+\mathcal{O}(\|E\|^2),\quad \ell=0,1,2,\ldots
\end{equation}
Combining with \eqref{eqn317}, we present the following definition for the Fr\'{e}chet derivatives of $\varphi$-functions:
\begin{definition}
The Fr\'{e}chet derivatives of $\varphi$-functions at $A$ in the direction $E$ is given by
\begin{equation}\label{eqn477}
L_{\varphi_{\ell}}(A,E)=\int_{0}^{1}\exp\big((1-s)A\big)s^{\ell}E\varphi_{\ell}(sA)ds,\quad \ell=0,1,2,\ldots
\end{equation}
\end{definition}
As a result,
\begin{equation}\label{216}
\varphi_{\ell}(A+E)=\varphi_{\ell}(A)+L_{\varphi_{\ell}}(A,E)+o(\|E\|).
\end{equation}
In summary, the following theorem shows that the values of $\epsilon_{\rm col}$ and $\epsilon_{\rm row}$ used during
the sparse column-row approximation will have a direct impact upon the final accuracy of computing $\varphi_{\ell}(A)$. Note that $XTY^T$ can be {\it any} low-rank approximation to $A$ in this theorem.
\begin{theorem}\label{Thm2.4}
Let $XTY^T$ be a sparse column-row approximation to $A$, and $\|A-XTY^T\|=\varepsilon$. Then we have
\begin{equation}\label{eq2.16}
\|\varphi_{\ell}(A)-\varphi_{\ell}(XTY^T)\|\lesssim{\rm cond}_{\rm abs}(\varphi_{\ell},A)~\varepsilon,
\end{equation}
and
\begin{equation}\label{eq2.17}
\frac{\|\varphi_{\ell}(A)-\varphi_{\ell}(XTY^T)\|}{\|\varphi_{\ell}(A)\|}\lesssim{\rm cond}_{\rm rel}(\varphi_{\ell},A)\frac{\varepsilon}{\|A\|},
\end{equation}
where $\lesssim$ represents omitting the high order term $o(\|E\|)$.
\end{theorem}
\begin{proof}
Let $E=A-XTY^T$, then we get from \eqref{216} that
\begin{eqnarray*}
\|\varphi_{\ell}(A)-\varphi_{\ell}(XTY^T)\|&\lesssim&\|L_{\varphi_{\ell}}(A)\|\|E\|
=\|L_{\varphi_{\ell}}(A)\|\varepsilon\nonumber\\
&=&{\rm cond}_{\rm abs}(\varphi_{\ell},A)\varepsilon,
\end{eqnarray*}
in which the high order term $o(\|E\|)$ is omitted.
The upper bound \eqref{eq2.17} of the relative error is derived from \eqref{eqn330} and \eqref{eq2.16}.
\end{proof}

\section{New strategies for estimating the absolute and relative 2-norm condition numbers}
\setcounter{equation}{0}
By Theorem \ref{Thm2.4}, it is crucial to consider how to estimate the absolute and relative condition numbers ${\rm cond}_{\rm abs}(\varphi_{\ell},A)$ and ${\rm cond}_{\rm rel}(\varphi_{\ell},A)$ efficiently.
Notice that
\begin{equation}\label{eqn4.12}
{\rm vec}(ACB)=(B^T\otimes A){\rm vec}(C).
\end{equation}
Since $L_f$ is a linear operator, we have
\begin{equation}\label{eqn499}
{\rm vec}(L_f(A,E))=K_f(A){\rm vec}(E)
\end{equation}
for some $K_f(A)\in\mathbb{R}^{n^2\times n^2}$ that is independent of $E$. The matrix $K_f(A)$ is refered to as the Kronecker form of the Fr\'{e}chet derivative \cite[pp.60]{Higham}. Specifically, we have 
\begin{equation}\label{eqn4110}
\|L_f(A)\|_F=\lambda_{\max}\big(K_f(A)^T K_f(A)\big)^{1/2}=\|K_f(A)\|_2.
\end{equation}
To estimate $\|K_f(A)\|_2$, the power method can be applied \cite[Algorithm 3.20]{Higham}. However, the power method lacks convergence tests, and because of its linear convergence
rate the number of iteration required is unpredictable. Alternatively, the condition
number is based on the 1-norm \cite[Algorithm 3.22]{Higham}. Although there is no analogue to the relation (\ref{eqn4110}) for the 1-norm, the next result gives
a relation between $\|K_f (A)\|_1$ and $\|L_f (A)\|_1$.
\begin{theorem}\label{Prop1} {\rm\cite[Lemma 3.18]{Higham}}
For $A\in\mathbb{R}^{n\times n}$ and any function $f$,
\begin{equation}\label{eqn4111}
\frac{\|L_f(A)\|_1}{n}\leq \|K_f(A)\|_1\leq n\|L_f(A)\|_1.
\end{equation}
\end{theorem}
In practice, however, the 2-norm is a more widely used norm than the 1-norm, and the former is preferable for both theoretical analysis and computational purposes. For example, one of the most important properties of 2-norm is the unitary invariance \cite{GV}.
Thus, we focus on the 2-norm condition number instead of the 1-norm condition number in this section. The following theorem establishes a relationship between $\|K_f (A)\|_2$ and $\|L_f (A)\|_2$.
\begin{theorem}\label{Prop2}
For $A\in\mathbb{R}^{n\times n}$ and any function $f$,
\begin{equation}\label{eqn4112}
\frac{\|L_f(A)\|_2}{\sqrt{n}}\leq \|K_f(A)\|_2\leq \sqrt{n}\|L_f(A)\|_2.
\end{equation}
\end{theorem}
\begin{proof}
For any $M\in\mathbb{R}^{n\times n}$, we have
\begin{equation}\label{eqn4113}
\|M\|_2\leq \|M\|_F=\|{\rm vec}(M)\|_2\leq \sqrt{n}\|M\|_2.
\end{equation}
Hence, it is seen from (\ref{eqn499}) and (\ref{eqn4113}) that
\begin{equation}\label{eqn4114}
\frac{\|K_f(A){\rm vec}(E)\|_2}{\|{\rm vec}(E)\|_2}=\frac{\|{\rm vec}\big(L_f(A,E)\big)\|_2}{\|{\rm vec}(E)\|_2}\leq \frac{\sqrt{n}
\|L_f(A,E)\|_2}{\|E\|_2},\quad\forall E\in \mathbb{R}^{n\times n},~E\neq O.
\end{equation}
Similarly,
\begin{equation}\label{eqn4115}
\frac{\|K_f(A){\rm vec}(E)\|_2}{\|{\rm vec}(E)\|_2}\geq \frac{\|L_f(A,E)\|_2}{\sqrt{n}\|E\|_2},\quad\forall E\in \mathbb{R}^{n\times n},~E\neq O.
\end{equation}
Maximizing over all $E$ for  (\ref{eqn4114}) and (\ref{eqn4115}) yields (\ref{eqn4112}).
\end{proof}

Compared (\ref{eqn4112}) with (\ref{eqn4111}), we see that investigating the 2-norm condition number of $K_f(A)$ is preferable to investigating its 1-norm condition number.
We are ready to show how to efficiently evaluate the Fr\'{e}chet 2-condition numbers of $\varphi$-functions for large sparse, low-rank matrices or matrices with fast decaying singular values.
Two novel strategies are proposed to evaluate the absolute and relative condition numbers.

{\bf Strategy I.}~~The key idea of the first strategy is to relate $L_{\varphi_{\ell}}(A)$ to $\varphi_{1}(Z)$ and $\varphi_{\ell+1}(Z)$. We notice from (\ref{eqn477}) and (\ref{eqn4.12}) that
\begin{eqnarray}\label{eqn4.10}
{\rm vec}\big(L_{\varphi_{\ell}}(A,E)\big)&=&\int_{0}^{1}{\rm vec}\big(\exp\big((1-s)A\big)s^{\ell}E\varphi_{\ell}(sA)\big)ds\nonumber\\
&=&\int_{0}^{1}\big(\varphi_{\ell}(sA^T)\otimes\exp\big((1-s)A\big)s^{\ell}\big){\rm vec}(E)ds\nonumber\\
&=&\big(I\otimes\exp(A)\big)\int_{0}^{1}\big(\varphi_{\ell}(sA^T)\otimes\exp(-sA)\big)s^{\ell}ds~{\rm vec}(E).
\end{eqnarray}
Let $X=Q_1R_1,Y=Q_2R_2$ be the (sparse) QR decomposition of $X$ and $Y$, respectively, where $Q_1,Q_2\in\mathbb{R}^{n\times r}$ are orthonormal and $R_1,R_2\in\mathbb{R}^{r\times r}$ are upper triangular. Motivated by Theorem \ref{Thm4.1} and (\ref{eqn4.10}), in Strategy I we make use of
\begin{equation}\label{est11}
{\rm cond}_{\rm abs}^{\rm I}(\varphi_{\ell},\widetilde{A})=\big\|R_1\varphi_1(Z)TR_2^T\cdot R_1\varphi_{\ell+1}(Z)TR_2^T\big\|_2
\end{equation}
as an estimation to the absolute condition number ${\rm cond}_{\rm abs}(\varphi_{\ell},A)$.

Theorem \ref{Thm4.1} also provides a cheap way to estimate 2-norms of $\varphi_{\ell}(A),~\ell=0,1,\ldots,p$.
Indeed, we have from \eqref{4.1} that
\begin{equation}\label{eqn2.6}
\Big|\|R_1[\varphi_{\ell+1}(Z)T]R_2^T\|_2-1/\ell!\Big|\leq\|\varphi_{\ell}(\widetilde{A})\|_2\leq 1/\ell!+\|R_1[\varphi_{\ell+1}(Z)T]R_2^T\|_2.
\end{equation}
Thus, we can use
\begin{equation}\label{eqn25}
\eta_{\ell}=\big\|R_1[\varphi_{\ell+1}(Z)T]R_2^T\big\|_2,\quad \ell=0,1,\ldots,p,
\end{equation}
as approximations to $\|\varphi_{\ell}({A})\|_2$.
In view of (\ref{eqn25}), the relative condition number ${\rm cond}_{\rm rel}(\varphi_{\ell},A)$
can be approximated by using
\begin{equation}\label{est12}
{\rm cond}_{\rm rel}^{\rm I}(\varphi_{\ell},\widetilde{A})=\frac{\|A\|_2\|R_1\varphi_1(Z)TR_2^T\cdot R_1\varphi_{\ell+1}(Z)TR_2^T\|_2}{\|R_1[\varphi_{\ell+1}(Z)T]R_2^T\|_2}.
\end{equation}
Recall that there is no need to form and store the Q-factors $Q_1$ and $Q_2$ in practice.

{\bf Strategy II.}~~The key idea of the second strategy is to relate $L_{\varphi_{\ell}}(A)$ to $L_{\varphi_{\ell+1}}(Z)$. Recall that $\varphi_{\ell}(z)$ can be expressed as the following power series whose radius of convergence is $\infty$:
$$
\varphi_{\ell}(z)=\sum_{i
=0}^{\infty}\frac{z^i}{(i+\ell)!}.
$$
The following proposition can be viewed as a generalization of Theorem \ref{Thm4.1} to any matrix function in power series.
\begin{prop}
Suppose that the power series $f(z)=\sum_{i=0}^{\infty}\alpha_i z^i$ has radius of convergence $\rho$. Let $\widetilde{A}=XTY^T\in\mathbb{R}^{n\times n}$, where $X,Y\in\mathbb{R}^{n\times r}$ and $T\in\mathbb{R}^{r\times r}$. Let $g(z)=\sum_{i=1}^{\infty}\alpha_i z^{i-1}$ and suppose that $\|\widetilde{A}\|<\rho$, then
$$
f(\widetilde{A})=f(O)+Xg(Z)TY^T,
$$
where $Z=T(Y^TX)\in\mathbb{R}^{r\times r}$.
\end{prop}
\begin{proof}
It is seen that $\widetilde{A}^k=XZ^{k-1}TY^T,~k\geq1$. Thus,
\begin{eqnarray*}
f(\widetilde{A})&=&\alpha_0I+\alpha_1\widetilde{A}+\cdots+\alpha_k\widetilde{A}^k+\cdots\\
&=&\alpha_0I+X\big[\alpha_1I+\alpha_2Z+\cdots+\alpha_kZ^{k-1}+\cdots\big]TY^T\\
&=&f(O)+Xg(Z)TY^T.
\end{eqnarray*}
\end{proof}

The following theorem gives closed-form formulae for $K_f(\widetilde{A})$ and $K_g(Z)$.
\begin{theorem}\label{Thm5}
Under the above notations, we have
\begin{equation}\label{eqn4.13}
K_g(Z)=\sum_{i=2}^{\infty}\alpha_i\sum_{j=1}^{i-1}\big((Z^T)^{i-j-1}\otimes Z^{j-1}\big).
\end{equation}
Denote $W=YT^T$, then
\begin{equation}\label{eqn415}
K_f(\widetilde{A})=\Psi_1+\Psi_2+\Psi_3,
\end{equation}
where
$$
\Psi_1=\alpha_1 I\otimes I,
$$
$$
\Psi_2=\Big(W\otimes I\Big)\sum_{i=2}^{\infty}\alpha_i\big((Z^T)^{i-2}\otimes I\big)\Big(X\otimes I\Big)^T+
\Big(I\otimes X\Big)\sum_{i=2}^{\infty}\alpha_i\big(I\otimes Z^{i-2}\big)\Big(I\otimes W\Big)^T,
$$
and
$$
\Psi_3=\Big(W\otimes X\Big)\Big(\sum_{i=3}^{\infty}\alpha_i\sum_{j=2}^{i-1}\big((Z^T)^{i-j-1}\otimes Z^{j-2}\big)\Big)\Big(X\otimes W\Big)^T.
$$
\end{theorem}
\begin{proof}
It follows from (\ref{eqn38}) and the expression of $g(x)$ that
$$
L_g(Z,F)=\sum_{i=2}^{\infty}\alpha_i\sum_{j=1}^{i-1}Z^{j-1}FZ^{i-j-1},\quad \forall F\in\mathbb{R}^{r\times r}.
$$
By (\ref{eqn4.12}),
\begin{eqnarray*}
{\rm vec}\big(L_g(Z,F)\big)&=&\sum_{i=2}^{\infty}\alpha_i\sum_{j=1}^{i-1}{\rm vec}(Z^{j-1}FZ^{i-j-1})\nonumber\\
&=&\sum_{i=2}^{\infty}\alpha_i\sum_{j=1}^{i-1}\big((Z^T)^{i-j-1}\otimes Z^{j-1}\big){\rm vec}(F),
\end{eqnarray*}
so we get (\ref{eqn4.13}). Similarly, for any $E\in\mathbb{R}^{n\times n}$, we have
\begin{eqnarray}\label{eqn420}
L_f(\widetilde{A},E)&=&\sum_{i=1}^{\infty}\alpha_i\sum_{j=1}^{i}\widetilde{A}^{j-1}E\widetilde{A}^{i-j}\nonumber\\
&=&\alpha_1E+\alpha_2(E\widetilde{A}+\widetilde{A}E)+\sum_{i=3}^{\infty}\alpha_i\Big(E\widetilde{A}^{i-1}+\sum_{j=2}^{i-1}\widetilde{A}^{j-1}E\widetilde{A}^{i-j}+\widetilde{A}^{i-1}E\Big)\nonumber\\
&=&\alpha_1E+\sum_{i=2}^{\infty}\alpha_i(E\widetilde{A}^{i-1}+\widetilde{A}^{i-1}E)+\sum_{i=3}^{\infty}\alpha_i\sum_{j=2}^{i-1}\widetilde{A}^{j-1}E\widetilde{A}^{i-j}.
\end{eqnarray}
As a result,
\begin{equation}\label{eqn4.20}
{\rm vec}(\alpha_1E)=(\alpha_1 I\otimes I){\rm vec}(E)=\Psi_1 {\rm vec}(E),
\end{equation}
and we have from $\widetilde{A}^{i-1}=XZ^{i-2}TY^T=XZ^{i-2}W^T~(i\geq 2)$ that
\begin{eqnarray}\label{eqn4.21}
{\rm vec}\Big(\sum_{i=2}^{\infty}\alpha_i(E\widetilde{A}^{i-1}+\widetilde{A}^{i-1}E)\Big)&=&{\rm vec}\Big(\sum_{i=2}^{\infty}\alpha_i\big(EXZ^{i-2}W^T+XZ^{i-2}W^TE\big)\Big)\nonumber\\
&=&\sum_{i=2}^{\infty}\alpha_i\big(W(Z^T)^{i-2}X^T\otimes I+I\otimes XZ^{i-2}W^T\big){\rm vec}(E)\nonumber\\
&=&\Big(\sum_{i=2}^{\infty}\alpha_i(W\otimes I)\big((Z^T)^{i-2}\otimes I\big)(X\otimes I)^T\nonumber\\
&&+\sum_{i=2}^{\infty}\alpha_i(I\otimes X)\big(I\otimes Z^{i-2}\big)(I\otimes W)^T\big]\Big){\rm vec}(E)\nonumber\\
&=&\Psi_2{\rm vec}(E).
\end{eqnarray}
Moreover,
\begin{eqnarray}\label{eqn4.23}
{\rm vec}\Big(\sum_{i=3}^{\infty}\alpha_i\sum_{j=2}^{i-1}\widetilde{A}^{j-1}E\widetilde{A}^{i-j}\Big)&=&\sum_{i=3}^{\infty}\alpha_i\sum_{j=2}^{i-1}{\rm vec}\big(XZ^{j-2}W^T\cdot E\cdot XZ^{i-j-1}W^T\big)\nonumber\\
&=&\sum_{i=3}^{\infty}\alpha_i\sum_{j=2}^{i-1}\big((W(Z^T)^{i-j-1}X^T) \otimes (XZ^{j-2}W^T)\big){\rm vec}(E)\nonumber\\
&=&\sum_{i=3}^{\infty}\alpha_i\sum_{j=2}^{i-1}(W\otimes X)\big((Z^T)^{i-j-1}\otimes Z^{j-2}) (X\otimes W)^T{\rm vec}(E)\nonumber\\
&=&\Psi_3{\rm vec}(E),
\end{eqnarray}
and (\ref{eqn415}) follows from (\ref{eqn420})--(\ref{eqn4.23}).
\end{proof}

\begin{rem}
Theorem \ref{Thm5} indicates that $L_f(\widetilde{A})$ and $L_g(Z)$ are closely related. More precisely, let
$$
\phi_{i}(Z)=\sum_{j=2}^{i-1}\big((Z^T)^{i-j-1}\otimes Z^{j-2}\big),\quad i=3,4,\ldots
$$
then
$\Psi_3=\big(W\otimes X\big)\sum_{i=3}^{\infty}\alpha_i\phi_{i}(Z)\big(X\otimes W\big)^T$.
On the other hand, if we denote
$$
\psi_i(Z)=\sum_{j=1}^{i-1}\big((Z^T)^{i-j-1}\otimes Z^{j-1}\big), \quad i=2,3,\ldots
$$
then $K_g(Z)=\sum_{i=2}^{\infty}\alpha_i\psi_i(Z)$, and it is seen from {\rm(}\ref{eqn4.13}{\rm)} that
$$
\psi_{i-1}(Z)=\phi_{i}(Z),\quad i=3,4,\ldots
$$
\end{rem}

Let $\varphi_{\ell}(\widetilde{A})=\frac{1}{\ell!}I+X\big[\varphi_{\ell+1}(Z)T\big]Y^T$ and let $\widetilde{\varphi_{\ell}(\widetilde{A})}=\frac{1}{\ell!}I+X\big[\varphi_{\ell+1}(Z+F)T\big]Y^T$. Then we have from (\ref{eqn317}) that
$$
L_{\varphi_{\ell+1}}(Z,F)=\varphi_{\ell+1}(Z+F)-\varphi_{\ell+1}(Z)+o(\|F\|),\quad  \forall F\in\mathbb{R}^{r\times r}.
$$
Hence,
\begin{eqnarray*}
\widetilde{\varphi_{\ell}(\widetilde{A})}-\varphi_{\ell}(A)&=&X\big[\varphi_{\ell+1}(Z+F)-\varphi_{\ell+1}(Z)\big]TY^T\\
&=&XL_{\varphi_{\ell+1}}(Z,F)TY^T+o(\|F\|).
\end{eqnarray*}
Let $X=Q_1R_1,Y=Q_2R_2$ be the (sparse) QR decomposition of $X$ and $Y$, respectively, where $R_1,R_2\in\mathbb{R}^{r\times r}$. Inspired by Theorem \ref{Thm5},
in Strategy II we make use of
\begin{equation}\label{Eqn4.11}
{\rm cond}_{\rm abs}^{\rm II}(\varphi_{\ell},\widetilde{A})=\|XL_{\varphi_{\ell+1}}(Z)TY^T\|_2=\|R_1L_{\varphi_{\ell+1}}(Z)TR_2^T\|_2
\end{equation}
as an approximation to the absolute 2-condition number ${\rm cond}_{\rm abs}(\varphi_{\ell},A)$.
And the relative 2-condition number ${\rm cond}_{\rm rel}(\varphi_{\ell},A)$
can be approximated by
\begin{equation}\label{Eqn4112}
{\rm cond}_{\rm rel}^{\rm II}(\varphi_{\ell},\widetilde{A})=\frac{\|A\|_2\|R_1L_{\varphi_{\ell+1}}(Z)TR_2^T\|_2}{\|R_1[\varphi_{\ell+1}(Z)T]R_2^T\|_2}.
\end{equation}
Similar to Strategy I, there is no need to form and store $Q_1$ and $Q_2$, and the key is to
evaluate 2-norms of some $r$-by-$r$ matrices.

\section{Numerical experiments}

\setcounter{equation}{0}

In this section, we perform some numerical experiments to illustrate the numerical behavior of our new method.
All the numerical experiments were run
on a Dell PC with eight cores Intel(R) Core(TM)i7-2600 processor with CPU
3.40 GHz and RAM 16.0 GB, under the Windows 7 with 64-bit operating system. All the numerical results were obtained from MATLAB R2015b implementations
with machine precision $\epsilon\approx 2.22\times 10^{-16}$.

In all the examples, the
sparse column-row approximation of $A$ is computed by using the MATLAB functions {\tt scra.m} and {\tt spqr.m} due to G.W. Stewart\footnote{\it ftp://ftp.cs.umd.edu/pub/stewart/reports/Contents.html.}, where the tolerance {\tt tol} is taken as $\epsilon_{\rm col}=\epsilon_{\rm row}=10^{-5}$.
In order to estimate the rank of a matrix, we consider the {structural rank} of $A$, i.e., {\tt sprank$(A)$} that is obtained from running the MATLAB built-in function {\tt sprank.m}.
The matrix exponential is calculated by using the MATLAB built-in function {\tt expm.m}, while the $\varphi_{\ell}(\ell\geq 1)$ functions are computed by using
the {\tt phipade.m} function of the MATLAB package EXPINT \cite{Ber2}.


\subsection{An application to data dimensionality reduction}

In this example, we show efficiency of our new method for computing matrix exponentials of large scale and low-rank matrices.
Many data mining problems involve data
sets represented in very high-dimensional spaces. In order to handle
high dimensional data, the dimensionality needs to be
reduced.
Linear
discriminant analysis (LDA) is one of notable subspace transformation
methods for dimensionality reduction \cite{DHS}.
LDA encodes discriminant
information by maximizing the between-class scatter,
and meanwhile minimizing the within-class scatter in
the projected subspace. Let $X=[{\bf a}_{1},{\bf a}_{2},\ldots,{\bf a}_{m}]$ be a set of training samples in an $n$-dimensional feature space, and
assume that
the original data is partitioned into $K$ classes as
$X=[X_1,X_2,\ldots,X_K]$.
We denote by $m_{j}$ the number of samples in the $j$-th class, and thus $\sum_{j=1}^{K}m_{j}=m$. Let ${\bf c}_{j}$ be the centroid of the $j$-th class, and ${\bf c}$ be the global centroid of the training data set.
If we denote ${\bf e}_{j}=[1,1,\ldots,1]^{T} \in \mathbb{R}^{m_{j}}$, then the within-class scatter matrix is defined as
$$
S_{W}=\sum_{j=1}^{K}\sum_{{\bf a}_{i}\in X_{j}}({\bf a}_{i}-{\bf c}_{j})({\bf a}_{i}-{\bf c}_{j})^{T}=H_WH_W^T,
$$
where
$H_{W}=[X_{1}-{\bf c}_{1}{\bf e}_{1}^{T},\ldots,X_{K}-{\bf c}_{K}{\bf e}_{K}^{T}]\in\mathbb{R}^{n\times m}$.
The between-class scatter matrix is defined as
$$
S_{B}=\sum_{j=1}^{K}n_{j}({\bf c}_{j}-{\bf c})({\bf c}_{j}-{\bf c})^{T}=H_BH_B^T,
$$
where
$H_B=[\sqrt{n_{1}}({\bf c}_{1}-{\bf c}),\sqrt{n_{2}}({\bf c}_{2}-{\bf c}),\ldots,\sqrt{n_{K}}({\bf c}_{K}-{\bf c})]\in\mathbb{R}^{n \times K}$.
The LDA method is realized by maximizing the between-class scatter distance while minimizing the total scatter distance, and the optimal projection matrix can be obtained from solving the following large scale generalized eigenproblem
\begin{equation}\label{1.3}
S_{B}{\bf x}= \lambda S_{W}{\bf x}.
\end{equation}

However, the dimension of real data usually exceeds the number of training
samples in practice (i.e., $n\gg m$),
which results in $S_W$ and $S_B$ being singular.
Indeed, suppose that the
training vectors are linearly independent, then the rank of $S_B$ and $S_W$
is $K-1$ and $m-K$, respectively, which is much smaller than the dimensionality $n$ \cite{DHS}.
This is called the small-sample-size (SSS) or undersampled
problem \cite{DHS,PP}. It is an intrinsic limitation of the classical
LDA method, and is also a common problem in classification
applications \cite{PP}.
In other words, the SSS problem stems from generalized
eigenproblems with singular matrices.
So as to cure this drawback, a novel method based on matrix exponential, called exponential discriminant analysis
method (EDA), was proposed in \cite{ZF}. Instead of (\ref{1.3}), the EDA method solve the following {\it generalized matrix exponential eigenproblem} \cite{ZF}
\begin{equation}\label{eqn3.12}
\textrm{exp}(S_{B}){\bf x}= \lambda \textrm{exp}(S_{W}){\bf x}.
\end{equation}
The EDA method is described as follows, for more details, refer to \cite{ZF}.
\begin{algorithm}{\rm\cite{ZF} \label{Alg1} {\bf The exponential discriminant analysis method (EDA)}}\\
\textbf{Input:} The data matrix $X=[{\bf a}_{1},{\bf a}_{2},\ldots,{\bf a}_{m}]\in\mathbb{R}^{n\times m}$, where ${\bf a}_{j}$ represernts the $j$-th training image.\\
\textbf{Output:} The projection matrix $V$.\\
1.~Compute the matrices $S_{B}$, $S_{W}$, ${\rm exp}(S_{B})$, and ${\rm exp}(S_{W})$;\\
2.~Compute the eigenvectors $\{{\bf x}_{i}\}$ and
eigenvalues $\{\lambda_{i}\}$ of ${\rm exp}(S_{W})^{-1}{\rm exp}(S_{B})$;\\
3.~Sort the eigenvectors $V=\{{\bf x}_{i}\}$
 according to $\lambda_{i}$ in decreasing order;\\
4.~Orthogonalize the columns of the projection matrix $V$.
\end{algorithm}

As both $\textrm{exp}(S_{W})$ and $\exp(S_B)$ are symmetric positive definite (SPD), the difficulty of SSS problem can be cured naturally in the EDA method. The framework of the EDA method for dimensionality reduction has gained wide attention in recent years \cite{Ah,DB,WCP,Wang,YPan,ZF}.
However, the time complexity of EDA is dominated by
the computation of ${\rm exp}(S_B)$ and ${\rm exp}(S_W)$, which is prohibitively large as data dimension is large \cite{ZF}.
By Theorem \ref{Thm4.1}, we can compute the large matrix exponentials as follows:

\begin{Cor}\label{Cor3.2}
Under the above notations, we have that
\begin{equation}\label{eqn314}
\exp(S_B)=I+H_B\big[\varphi_1(H_B^TH_B)\big] H_B^T,
\end{equation}
and
\begin{equation}\label{eqn315}
\exp(S_W)=I+H_W\big[\varphi_1(H_W^TH_W)\big]H_W^T.
\end{equation}
\end{Cor}

So we have the following algorithm for the matrix exponential discriminant analysis method.

\begin{algorithm}\label{Alg.2} {\bf An algorithm for computing $\exp(S_B)$ and $\exp(S_W)$}\\
1.~Given the data matrix $X=[{\bf a}_{1},{\bf a}_{2},\ldots,{\bf a}_{m}]\in\mathbb{R}^{n\times m}$, form $H_B$ and $H_W$;\\
2.~Compute $\varphi_1(H_B^TH_B)$ and $\varphi_1(H_W^TH_W)$;\\
3.~Store $H_B,H_W$ and $\varphi_1(H_B^TH_B)$, $\varphi_1(H_W^TH_W)$ for $\exp(S_B)$ and $\exp(S_W)$. If desired, form $\exp(S_B)$ and $\exp(S_W)$ in terms of \eqref{eqn314} and \eqref{eqn315}.
\end{algorithm}

Note that both $H_B$ and $H_W$ are already available in Step 1, so there is no need to perform
rank-revealing decompositions to $S_B$ and $S_W$. As a result, the computation of the two $n\times n$ matrix exponentials $\exp(S_B),\exp(S_W)$ reduces to that of $\varphi_1(H_B^TH_B)\in\mathbb{R}^{K\times K}$ and $\varphi_1(H_W^TH_W)\in\mathbb{R}^{m\times m}$, with $K,m\ll n$.

Next we illustrate the efficiency of Algorithm \ref{Alg.2} for the matrix exponential discriminant analysis method. There are three real-world databases in this example.
The first one is the {\tt ORL} database\footnote{\it http://www.cl.cam.ac.uk/research/dtg/attarchive/facedatabase.html.} that contains 400 face images of $40$
individuals, and the original image size is $92\times 112=10304$.
The second test set is the {\tt Yale} face database taken from
the Yale Center for Computational Vision and Control\footnote{\it http://vision.ucsd.edu/datasets/yale\_face\_dataset\_original/yalefaces.zip.}.
It contains $165$ grayscale images of $K=15$ individuals. The original image
size is $320\times 243=77760$.
The third test set is the {\tt Extended YaleB} database\footnote{\it http://vision.ucsd.edu/\~~leekc/ExtYaleDatabase/Yale\%20Face\%20Database.html.}.
This database contains 5760 single light source images of 10 subjects, each seen under 576 viewing conditions.
A subset of $38$ classes with 2432 images are used in this example, 64 images of per individual with illumination.

In the {\tt ORL} database, the images are
aligned based on eye coordinates and are cropped and scaled to
$n=32\times 32$ and $64\times 64$, respectively; and the original image size with $n=92\times 112$ is also considered.
In the {\tt Yale} and the {\tt Extended YaleB} databases, all images are
aligned based on eye coordinates and are cropped and scaled to
$n=32\times 32,~64\times 64$ and $100\times 100$, respectively.
In this example, a random subset with $3$ images per subject is
taken to form the training set, and the rest of the images are
used as the testing set. Each column of the data matrices is scaled by its 2-norm.

In Algorithm \ref{Alg.2}, the CPU time consists of computing $H_B,H_W$, evaluating $\varphi_1(H_B^TH_B)$ and $\varphi_1(H_W^TH_W)$, as well as forming $\exp(S_B)$ and $\exp(S_W)$ in terms of (\ref{eqn314}) and (\ref{eqn315}). 
In the original EDA algorithm (Algorithm \ref{Alg1}), the CPU time consists of forming $H_B,H_W$, and the computation of $\exp(S_B)$ and $\exp(S_W)$ by using the MATLAB built-in function {\tt expm.m}.

Let $\exp(S_B),\exp(S_W)$ be the ``exact solutions" obtained from running {\tt expm.m},
and let $\widetilde{\exp(S_B)},\widetilde{\exp(S_W)}$ be the approximations obtained from (\ref{eqn314}) and (\ref{eqn315}).
In this example, we define
$$
{\bf Err_B}=\frac{\|\exp(S_B)-\widetilde{\exp(S_B)}\|_F}{\|\exp(S_B)\|_F},\quad {\bf Err_W}=\frac{\|\exp(S_W)-\widetilde{{\rm exp}(S_W)}\|_F}{\|\exp(S_W)\|_F},
$$
as the relative errors of the approximations $\widetilde{{\rm exp}(S_B)},\widetilde{{\rm exp}(S_W)}$, respectively, and denote by
$$
{\bf Rel\_ErrF}=\max({\bf Err_B},{\bf Err_W})
$$
the maximal value of the two relative errors.
Table 1 lists the CPU time in seconds of Algorithm \ref{Alg.2}, {\tt expm.m}, and the values of the maximal relative errors {\tt Rel\_ErrF}.

{\small
\begin{table}[!h]
\begin{center}
\def\temptablewidth{0.9\textwidth}
{\rule{\temptablewidth}{0.8pt}}
\begin{tabular*}{\temptablewidth}{@{\extracolsep{\fill}}cccccr}
{\bf Database}   &$n$ &{\bf Algorithm \ref{Alg.2}}       &{\bf expm.m} & ${\bf Rel\_ErrF}$    \\\hline
ORL          & $1024$  & 0.08 &0.26 & $2.20\times 10^{-15}$  \\
             & $4096$  & 0.26 &17.3 & $3.22\times 10^{-15}$   \\
             & $10304$  &1.28 &261.1 &$4.49\times 10^{-15}$   \\\hline
Yale         & $1024$  & 0.08 &0.25 & $2.11\times 10^{-15}$  \\
             & $4096$  & 0.24 &17.3 & $3.10\times 10^{-15}$   \\
             & $10000$ &1.13 &238.5 &$4.35\times 10^{-15}$   \\\hline
Extended YaleB        & $1024$  & 0.08 &0.27 & $2.13\times 10^{-15}$  \\
             & $4096$  & 0.27 &17.3 & $3.14\times 10^{-15}$   \\
             & $10000$  &1.22 &238.6 &$4.50\times 10^{-15}$   \\
 \end{tabular*}
 {\rule{\temptablewidth}{1pt}}\\
 \end{center}
 \begin{flushleft}
  {\small {\rm {Example 1, Table 1}:~CPU time in seconds and the relative errors for computing $\exp(S_B)$ and $\exp(S_W)$.}}
 \end{flushleft}
 \end{table}
}

We observe from Table 1 that Algorithm \ref{Alg.2} runs much faster than {\tt expm.m}, especially when $n$ is large.
For instance, when the dimensionality of the datasets is around $10^4$, {\tt expm.m} requires about 240 seconds, while our new method only needs about
1.2 seconds, a great improvement. Furthermore, the relative errors of our approximations are in the order of $\mathcal{O}(10^{-15})$, implying that our new method
is numerically stable. Thus, the new method is very efficient and reliable for solving large matrix exponential problems arising in the EDA framework for high dimensionality reduction.

\subsection{Computing $\varphi$-functions of matrices with low rank or fast decaying singular values}

In this example, we show the efficiency of Algorithm \ref{Alg.1} for {\it consecutively} computing several $\varphi$-functions of $A$ with low rank or with fast decaying singular values. The test matrices are available from \cite{Testmatrix,Network}, and Table 2 lists problem characteristics of these matrices. Here the first five matrices are rank-deficient while the last three are full rank but with fast decaying singular values.

In this example, we compare Algorithm \ref{Alg.1} with {\tt expm.m/phipade.m}, that is, {\tt expm.m} for the matrix exponential $\exp(A)$ and {\tt phipade.m} for $\varphi_{\ell}(A),~\ell=1,2,3,4$.
In Algorithm \ref{Alg.1}, the CPU time consists of computing the sparse column-row approximation (SCR), the evaluation of $\varphi_{\ell}(Z)~(\ell=1,2,3,4,5)$ by using {\tt phipade.m}, as well as forming $\varphi_{\ell}(A)$ in terms of (\ref{4.1}), $\ell=0,1,2,3,4$. In {\tt expm.m}/{\tt phipade.m}, the CPU time consists of computing $\varphi_{\ell}(A)$ by using {\tt expm.m} ($\ell=0$) and {\tt phipade.m}, $\ell=1,2,3,4$.
In order to measure the accuracy of the computed solutions, we define the maximal relative error as
$$
{\bf Rel\_ErrF}=\max_{0\leq\ell\leq 4}\frac{\|\varphi_{\ell}(A)-\widetilde{\varphi_{\ell}(\widetilde{A})}\|_F}{\|\varphi_{\ell}(A)\|_F},
$$
where $\varphi_{\ell}(A)$ is the ``exact solution" obtained from {\tt expm.m} as $\ell=0$ and {\tt phipade.m} as $\ell=1,2,3,4$; and $\widetilde{\varphi_{\ell}(\widetilde{A})}$ is the approximation obtained from running Algorithm \ref{Alg.1}. Table 3 lists the numerical results.


{\small
\begin{table}[!h]
\begin{center}
\def\temptablewidth{1\textwidth}
{\rule{\temptablewidth}{0.8pt}}
\begin{tabular*}{\temptablewidth}{@{\extracolsep{\fill}}lcccr}
{\bf Test matrix}  &$n$   &{\bf sprank$(A)$}  &{\bf nnz$(A)$}  &{\bf Description}\\\hline
     man5976       & 5976 &5882               &225046     & Structural problem            \\
     Movies        & 5757 &1275               &24451         & Directed network             \\
     lock3491      & 3491 &3416               &160444      &  Structural problem          \\
     cegb3306      & 3306 &3222               &74916      & Finite element framework             \\
     zenios        & 2873 &266                & 1314           &  Optimization problem  \\
     watt\_1      & 1856 &1856               & 11360            & Computational fluid dynamics         \\
     watt\_2     & 1856 &1856               &11550              & Computational fluid dynamics        \\
     eris1176     & 1176 &1176               & 18552  &  Power network problem            \\

 \end{tabular*}
 {\rule{\temptablewidth}{1pt}}\\
 \end{center}
 \begin{flushleft}
  {\small {\rm {Example 2, Table 2}:~Problem characteristics of the test matrices, where ``nnz$(A)$" denotes the number of nonzero elements of $A$.}}
 \end{flushleft}
 \end{table}
}

\begin{center}
\scalebox{0.5}{\includegraphics{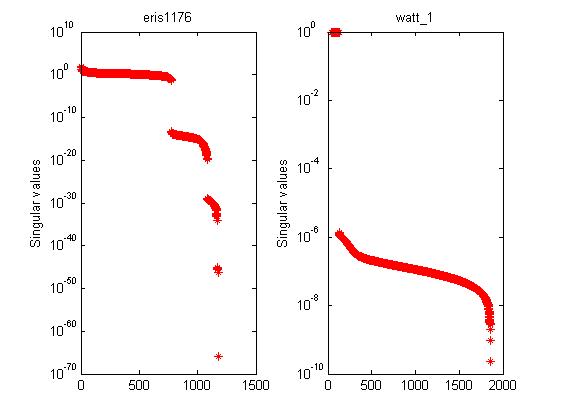}}\\
\end{center}
\begin{flushleft}
{\small Example 2, Figure 1:~Singular values of the {\tt eris1176} matrix and the {\tt watt\_1} matrix.}
\end{flushleft}

\vspace{0.5cm}

{\small
\begin{table}[!h]
\begin{center}
\def\temptablewidth{1\textwidth}
{\rule{\temptablewidth}{0.8pt}}
\begin{tabular*}{\temptablewidth}{@{\extracolsep{\fill}}lcccccr}
{\bf Test matrix}    &{\bf Algorithm 1} &{\bf expm.m}/{\bf phipade.m}  & {\bf Rel\_ErrF} \\\hline
     man5976$^{*}$               &205.5                     &10314.0          &$1.10\times 10^{-12}$  \\
     Movies                      &360.6                      &456.1           &$4.32\times 10^{-14}$  \\
     lock3491$^{*}$              &5.49                        &2469.6          &$5.12\times 10^{-13}$  \\
     cegb3306$^{*}$              &3.11                        &1599.7             &$1.24\times 10^{-13}$  \\
     zenios                      & 0.63                      &3.63          &$1.17\times 10^{-7}$  \\
     watt\_1                    & 0.19                      &30.4          &$1.93\times 10^{-7}$  \\
     watt\_2                    &0.19                     &39.1         &$2.02\times 10^{-7}$  \\
     eris1176$^{*}$             & 6.74                      &85.7          &$2.86\times 10^{-12}$  \\
 \end{tabular*}
 {\rule{\temptablewidth}{1pt}}\\
 \end{center}
 \begin{flushleft}
  {\small {\rm {Example 2, Table 3}:~CPU time in seconds and the maximal relative error for computing $\varphi_{\ell}$ matrix functions, $\ell=0,1,2,3,4$; where ``$^{*}$" denotes we compute $\varphi_{\ell}(-A)$ instead of $\varphi_{\ell}(A)$.}}
 \end{flushleft}
 \end{table}
}

It is seen from Table 3 that Algorithm \ref{Alg.1} works quite well, and we provide a competitive candidate for consecutively computing several $\varphi$-functions of large sparse matrices with low rank or with fast decaying singular values.
Firstly, Algorithm \ref{Alg.1} runs much faster than {\tt expm.m}/{\tt phipade.m}. For instance, 205.5 seconds vs. 10314 seconds for the {\tt man5976} matrix, 5.49 seconds vs. 2469.6 seconds for the {\tt lock3491} matrix, and 3.11 seconds vs. 1599.7 seconds for
the {\tt cegb3306} matrix. Secondly, the accuracy of our approximations is satisfactory in most cases.
However, for the {\tt zenios}, {\tt watt\_1} and {\tt watt\_2} matrices, the relative errors {\tt Rel\_ErrF} are in the order of $\mathcal{O}(10^{-7})$.
In Figure 1, we plot the singular values of the {\tt eris1176} matrix and the {\tt watt\_1} matrix. It is observed that the {\tt eris1176} matrix has faster decaying singular values, while the decaying speed of the singular values of the {\tt watt\_1} matrix is relatively slower.
Indeed, the error $\|A-XTY^T\|_F$ from the SCR decomposition, with respect to the three matrices {\tt zenios}, {\tt watt\_1} and {\tt watt\_2} are about $7.75\times 10^{-6}$, $9.97\times 10^{-6}$ and $9.98\times 10^{-6}$, respectively, while that of the {\tt eris1176} matrix is about $2.49\times 10^{-10}$.
In terms of Theorem \ref{Thm2.4}, the accuracy of the computed solution of the {\tt eris1176} matrix can be higher than that of {\tt zenios}, {\tt watt\_1} and {\tt watt\_2}, provided that the condition numbers are comparable.
Thus, our new method is suitable to $\varphi$-functions of large matrices with low rank or with fast decaying singular values.


\subsection{Estimating the relative and absolute condition numbers}

In this example, we demonstrate the efficiency of Strategy I and Strategy II for estimating the absolute and relative condition numbers of $\varphi$-functions.
There are four test matrices in this example, which are available from \cite{Testmatrix,Network}. The problem characteristics of these matrices are given in Table 4.

We compare Strategy I and Strategy II with the {\tt funm\_condest1.m} function in the Matrix Function Toolbox \cite{Toolbox}.
In Strategy I and Strategy II, the matrices $R_1,R_2$ are obtained from the QR decompositions of $X$ and $Y$, respectively, by the MATLAB built-in function {\tt qr.m}. The matrix $L_{\varphi_{\ell+1}}(Z)$ in \eqref{Eqn4112} is calculated by using the {\tt funm\_condest\_fro.m} function in the Matrix Function Toolbox.
When $\ell=0$, the parameter ``{\tt fun}" in {\tt funm\_condest1.m} is called by the MATLAB built-in function {\tt expm.m}, while $\ell>0$ this parameter is called by the {\tt phipade.m} function in the EXPINT package. The CPU time for both Strategy I and Strategy II is composed of computing the sparse column-row approximation, and calculating (\ref{est11}), (\ref{est12}) or (\ref{Eqn4.11}), (\ref{Eqn4112}), respectively. In \eqref{est12} and \eqref{Eqn4112}, $\|A\|_2$ is evaluated by using the MATLAB built-in function {\tt svds.m}. Tables 5--8 report the numerical results, where ``{\tt Relative\_Est}" and ``{\tt Absolute\_Est}" denote an estimation to the relative and the absolute condition number, respectively.

{\small
\begin{table}[!h]
\begin{center}
\def\temptablewidth{0.8\textwidth}
{\rule{\temptablewidth}{0.8pt}}
\begin{tabular*}{\temptablewidth}{@{\extracolsep{\fill}}lccr}
{\bf Network matrix}  &$n$ &{\bf sprank$(A)$}  &{\bf nnz$(A)$}     \\\hline
California & 9664 &1686 & $16150$   \\
EVA        & 8497 &1303 & $6726$   \\
EPA      & 4772 &986  & $8965$  \\
eris1176   &1176 &1176  & 18552   \\
 \end{tabular*}
 {\rule{\temptablewidth}{1pt}}\\
 \end{center}
 \begin{flushleft}
  {\small {\rm {Example 3, Table 4}:~Problem characteristics of the test matrices, where ``nnz$(A)$" denotes the number of nonzero elements of $A$.}}
 \end{flushleft}
 \end{table}
}

As was pointed out in \cite[pp.64]{Higham}, for the absolute and relative condition numbers, ``what is needed is an estimate that is of the correct order of magnitude in practice---more than one correct significant digit is not needed".
Recall that the {\tt funm\_condest1.m} function estimates the 1-norm relative and absolute condition numbers, while Strategy I and Strategy II estimate the 2-norm condition numbers.
Compared with the numerical results of {\tt funm\_condest1}, we see from Tables 5--8 that both Strategy I and Strategy II capture the correct order of magnitude of the condition numbers in many cases, and we can not tell which one, Strategy I or Strategy II, is {\it definitely} better than the other. We find that Strategy I runs (a little) faster than Strategy II in terms of CPU time. The reason is that we have to evaluate $L_{\varphi_{\ell+1}}(Z)$ iteratively via the {\tt funm\_condest\_fro.m} function.

On the other hand, it is observed from the numerical results that our new strategies often run much faster than {\tt funm\_condest1}. For instance, as $\ell=0$, the new methods used 752.0 and 764.2 seconds for the {\tt California} matrix, respectively, while {\tt funm\_condest1} used 1477.5 seconds. As $\ell=1$, the new methods used 757.9 and 788.7 seconds, respectively, while {\tt funm\_condest1} used 5266.8 seconds. The improvement is impressive. Specifically, as $\ell\geq 2$, for some large matrices such as {\tt California} and {\tt EVA}, {\tt funm\_condest1} fails to converge within 3 hours. As a comparison, our new methods work quite well.
Thus, we benefit from our new strategies, and provide competitive alternatives for estimating the relative and absolute condition numbers of $\varphi$-functions with respect to large sparse matrices.

{\small
\begin{table}[!h]
\begin{center}
\def\temptablewidth{1\textwidth}
{\rule{\temptablewidth}{0.8pt}}
\begin{tabular*}{\temptablewidth}{@{\extracolsep{\fill}}lcccccr}
$\ell$  &{\bf Method} &{\bf Relative\_Est}  &{\bf Absolute\_Est}  &{\bf CPU}    \\\hline
0       & {\tt funm\_condest1}       &$6.82\times 10^3$ & $7.53\times 10^{4}$ &1477.5  \\
       & Strategy I              &$4.01\times 10^2$ & $3.92\times 10^{4}$  &752.0    \\
       & Strategy II             &$21.8$ & $2.13\times 10^{3}$ &764.2   \\\hline
1      & {\tt funm\_condest1}       &$1.33\times 10^4$ & $1.98\times 10^{4}$ &$5266.8$  \\
       & Strategy I              &$6.53\times 10^2$ & $8.58\times 10^{3}$  &757.9    \\
       & Strategy II             &$25.1$ & $3.29\times 10^{2}$ &788.7   \\\hline
2       & {\tt funm\_condest1}       &$-$ & $-$ &$>$3h  \\
       & Strategy I              &$1.38\times 10^3$ & $2.41\times 10^3$  &758.1    \\
       & Strategy II             &15.6 & 27.2 &801.0   \\\hline
3       & {\tt funm\_condest1}      &$-$ & $-$ &$>$3h  \\
       & Strategy I              &$2.57\times 10^3$ & $5.80\times 10^2$  &757.9    \\
       & Strategy II             &34.7 & 7.82 &827.0   \\\hline
4       & {\tt funm\_condest1}       &$-$ & $-$ &$>$3h   \\
       & Strategy I              &$4.10\times 10^3$ & $1.15\times 10^2$  &761.1    \\
       & Strategy II             &82.6 &2.31 &852.1   \\
 \end{tabular*}
 {\rule{\temptablewidth}{1pt}}\\
 \end{center}
 \begin{flushleft}
  {\small {\rm {Example 3, Table 5}:~Estimation of the relative and absolute condition numbers of $\varphi_{\ell}(A),~\ell=0,1,2,3,4$, and the CPU time in seconds, where ``$>$3h" denotes the algorithm fails to converge within 3 hours. The {\tt California} matrix, $n=9664$, {\tt sprank(A)}=1686.}}
 \end{flushleft}
 \end{table}
}

{\small
\begin{table}[!h]
\begin{center}
\def\temptablewidth{1\textwidth}
{\rule{\temptablewidth}{0.8pt}}
\begin{tabular*}{\temptablewidth}{@{\extracolsep{\fill}}lcccccr}
$\ell$  &{\bf Method} &{\bf Relative\_Est}  &{\bf Absolute\_Est}  &{\bf CPU}    \\\hline
0       & {\tt funm\_condest1}       &$7.74\times 10^3$ & $1.83\times 10^{4}$ &700.0  \\
       & Strategy I              &$3.53\times 10^2$ & $4.37\times 10^{2}$  &412.0    \\
       & Strategy II             &$1.07\times 10^3$ & $1.32\times 10^{3}$ &420.1   \\\hline
1      & {\tt funm\_condest1}       &$7.07\times 10^3$ & $7.22\times 10^{3}$ &$5846.1$  \\
       & Strategy I              &$3.17\times 10^2$ & $1.59\times 10^{2}$  &419.4    \\
       & Strategy II             &$5.38\times 10^2$ & $2.69\times 10^{2}$ &425.2   \\\hline
2       & {\tt funm\_condest1}       &$-$ & $-$ &$>$3h  \\
       & Strategy I              &$2.72\times 10^2$ & 45.3  &415.2    \\
       & Strategy II             &$3.21\times 10^2$ & 53.5 &429.8   \\\hline
3       & {\tt funm\_condest1}      &$-$ & $-$ &$>$3h  \\
       & Strategy I              &$2.49\times 10^2$ & 10.4  &408.9    \\
       & Strategy II             &$1.22\times 10^2$ & 5.08 &440.9   \\\hline
4       & {\tt funm\_condest1}       &$-$ & $-$ &$>$3h   \\
       & Strategy I              &$2.36\times 10^2$ & 1.97  &408.0    \\
       & Strategy II             &73.3 &0.61 &447.1   \\
 \end{tabular*}
 {\rule{\temptablewidth}{1pt}}\\
 \end{center}
 \begin{flushleft}
  {\small {\rm {Example 3, Table 6}:~Estimation of the relative and absolute condition numbers of $\varphi_{\ell}(A),~\ell=0,1,2,3,4$, and the CPU time in seconds, where ``$>$3h" denotes the algorithm fails to converge within 3 hours. The {\tt EVA} matrix, $n=8497$, {\tt sprank(A)}=1303.}}
 \end{flushleft}
 \end{table}
}

\vspace{0.5cm}

{\small
\begin{table}[!h]
\begin{center}
\def\temptablewidth{1\textwidth}
{\rule{\temptablewidth}{0.8pt}}
\begin{tabular*}{\temptablewidth}{@{\extracolsep{\fill}}lcccccr}
$\ell$  &{\bf Method}               &{\bf Rel\_Est}  &{\bf Abs\_Est}  &{\bf CPU}    \\\hline
0       & {\tt funm\_condest1}       &$9.90\times 10^{3}$ & $2.60\times 10^{4}$ &90.1  \\
       & Strategy I              &$2.48\times 10^2$ & $1.12\times 10^{3}$  &90.6    \\
       & Strategy II             &$3.34\times 10^4$ & $1.52\times 10^{5}$ &95.4   \\\hline
1       & {\tt funm\_condest1}       &$1.09\times 10^4$ & $8.05\times 10^{3}$ &272.1  \\
       & Strategy I              &$2.51\times 10^2$ & $3.06\times 10^{2}$  &90.3    \\
       & Strategy II             &$2.25\times 10^4$ & $2.74\times 10^{4}$ &100.6   \\\hline
2       & {\tt funm\_condest1}       &$8.95\times 10^3$ & $2.03\times 10^{3}$ &443.7  \\
       & Strategy I              &$2.65\times 10^2$ & 76.9  &90.8    \\
       & Strategy II             &$1.47\times 10^4$ & $4.27\times 10^3$ &105.5   \\\hline
3       & {\tt funm\_condest1}       &$7.86\times 10^3$ & $4.25\times 10^2$ &773.0  \\
       & Strategy I              &$2.85\times 10^2$ & 17.2  &92.5    \\
       & Strategy II             &$9.71\times 10^3$ & $5.87\times 10^2$ &110.7   \\\hline
4       & {\tt funm\_condest1}       &$7.17\times 10^3$ & 75.2 &1357.7  \\
       & Strategy I              &$3.03\times 10^2$ & 3.29  &91.7    \\
       & Strategy II             &$5.29\times 10^3$ & 57.6 &116.7   \\
 \end{tabular*}
 {\rule{\temptablewidth}{1pt}}\\
 \end{center}
 \begin{flushleft}
  {\small {\rm {Example 3, Table 7}:~Estimation of the relative and absolute condition numbers of $\varphi_{\ell}(A)$, $\ell=0,1,2,3,4$, and the CPU time in seconds. The {\tt EPA} matrix, $n=4772$, {\tt sprank(A)}=986.}}
 \end{flushleft}
 \end{table}
}

{\small
\begin{table}[!h]
\begin{center}
\def\temptablewidth{1\textwidth}
{\rule{\temptablewidth}{0.8pt}}
\begin{tabular*}{\temptablewidth}{@{\extracolsep{\fill}}lcccccr}
$\ell$  &{\bf Method} &{\bf Relative\_Est}  &{\bf Absolute\_Est}  &{\bf CPU}    \\\hline
0       & {\tt funm\_condest1}       &$575.1$ & $1.45\times 10^{3}$ &2.47  \\
       & Strategy I              &$1.10\times 10^4$ & $1.88\times 10^{4}$  &7.33    \\
       & Strategy II             &$2.72\times 10^3$ & $4.66\times 10^{3}$ &11.1  \\\hline
1      & {\tt funm\_condest1}       &$1.09\times 10^3$ & $548.7$ &38.2 \\
       & Strategy I              &$1.10\times 10^4$ & $3.69\times 10^{3}$  &7.48    \\
       & Strategy II             &$2.66\times 10^3$ & $890.5$ &14.5   \\\hline
2       & {\tt funm\_condest1}    &$1.75\times 10^3$ & $171.0$ &60.2  \\
       & Strategy I              &$1.10\times 10^4$ & $679.0$  &7.78    \\
       & Strategy II             &$2.71\times 10^3$ & $167.6$ &17.6   \\\hline
3       & {\tt funm\_condest1}    &$2.37\times 10^3$ &$42.2$ &82.2  \\
       & Strategy I              &$1.10\times 10^4$ & $114.9$  &8.09    \\
       & Strategy II             &$2.82\times 10^3$ & 29.4 &20.5   \\\hline
4       & {\tt funm\_condest1}       &$2.85\times 10^3$ & 8.49 &103.4  \\
       & Strategy I              &$1.10\times 10^4$ & $17.6$  &8.26    \\
       & Strategy II             &$2.88\times 10^3$ & 4.61 &24.2   \\
 \end{tabular*}
 {\rule{\temptablewidth}{1pt}}\\
 \end{center}
 \begin{flushleft}
  {\small {\rm {Example 3, Table 8}:~Estimation of the relative and absolute condition numbers of $\varphi_{\ell}(-A),~\ell=0,1,2,3,4$, and the CPU time in seconds. The {\tt eris1176} matrix, $n=1176$, {\tt sprank(A)}=1176.}}
 \end{flushleft}
 \end{table}
}

\subsection{Sharpness of Theorem \ref{Thm2.4}}

In this example, we aim to show the sharpness of Theorem \ref{Thm2.4}. The test matrix is the {\tt watt\_1} matrix used in Example 2; see Table 2. It is a
$1856\times 1856$ full-rank matrix with fast decaying singular values.
So as to show the sharpness of Theorem \ref{Thm2.4}, we denote by
$$
{\bf Abs\_Err2}=\|\varphi_{\ell}(A)-\varphi_{\ell}(\widetilde{A})\|_2,
$$
and
$$
{\bf Rel\_Err2}=\frac{\|\varphi_{\ell}(A)-\varphi_{\ell}(\widetilde{A})\|_2}{\|\varphi_{\ell}(A)\|_2},
$$
the absolute and relative errors of the computed solutions $\varphi_{\ell}(\widetilde{A})$ with respect to $\varphi_{\ell}(A)$ in terms of 2-norm.
The values of ${\rm cond}_{\rm abs}(\varphi_{\ell},A)$ and ${\rm cond}_{\rm rel}(\varphi_{\ell},A)$ in the upper bounds of \eqref{eq2.16} and \eqref{eq2.17} are estimated by Strategy I or Strategy II, respectively, and the corresponding estimations are denoted by ``\eqref{eq2.16}--StrI", ``\eqref{eq2.17}--StrI", and ``\eqref{eq2.16}--StrII", ``\eqref{eq2.17}--StrII", respectively. Table 9 lists the numerical results.

We see from Table 9 that both \eqref{eq2.16} and \eqref{eq2.17} are very sharp, which justify Strategy I and Strategy II for estimating ${\rm cond}_{\rm abs}(\varphi_{\ell},A)$ and ${\rm cond}_{\rm rel}(\varphi_{\ell},A)$. We find that the values of \eqref{eq2.16}--StrII and \eqref{eq2.17}--StrII are a little smaller than those of {\tt Abs\_Err2} and {\tt Rel\_Err2} in many cases. In fact, both Strategy I and Strategy II only give approximations to the absolute and relative condition numbers, which are neither upper bounds nor lower bounds theoretically.

{\small
\begin{table}[!h]
\begin{center}
\def\temptablewidth{1.05\textwidth}
{\rule{\temptablewidth}{0.8pt}}
\begin{tabular*}{\temptablewidth}{@{\extracolsep{\fill}}lcccccc}
  &$\ell=0$   &$\ell=1$  &$\ell=2$ &$\ell=3$  &$\ell=4$  \\\hline
{\bf Abs\_Err2} &$1.07\times 10^{-6}$ &$5.35\times 10^{-7}$&$1.78\times 10^{-7}$ &$4.45\times 10^{-8}$  &$8.91\times 10^{-9}$ \\
{\bf Rel\_Err2} &$3.93\times 10^{-7}$ &$3.11\times 10^{-7}$ &$2.48\times 10^{-7}$ &$2.04\times 10^{-7}$  &$1.73\times 10^{-7}$  \\\hline
\eqref{eq2.16}--StrI &$3.16\times 10^{-6}$ &$1.32\times 10^{-6}$ &$4.01\times 10^{-7}$  &$9.48\times 10^{-8}$  &$1.83\times 10^{-8}$\\
\eqref{eq2.17}--StrI &$1.84\times 10^{-6}$ &$1.84\times 10^{-6}$ &$1.84\times 10^{-6}$  &$1.84\times 10^{-6}$  &$1.84\times 10^{-6}$ \\\hline
\eqref{eq2.16}--StrII &$1.83\times 10^{-7}$ &$5.24\times 10^{-8}$ &$1.19\times 10^{-8}$ &$2.22\times 10^{-9}$  &$3.47\times 10^{-10}$  \\
\eqref{eq2.17}--StrII &$1.06\times 10^{-7}$ &$7.30\times 10^{-8}$ &$5.43\times 10^{-8}$ &$4.30\times 10^{-8}$  &$3.49\times 10^{-8}$  \\
\end{tabular*}
 {\rule{\temptablewidth}{1pt}}\\
 \end{center}
 \begin{flushleft}
  {\small {\rm Example 4, Table 9:~Absolute and relative errors and their estimations, $\ell=0,1,2,3,4$. The {\tt watt\_1} matrix}, with $\|A-\widetilde{A}\|_2\approx 1.07\times 10^{-6}$.
  Here \eqref{eq2.16}--StrI, \eqref{eq2.17}--StrI, \eqref{eq2.16}--StrII, \eqref{eq2.17}--StrII denote the values of ${\rm cond}_{\rm abs}(\varphi_{\ell},A)$ and ${\rm cond}_{\rm rel}(\varphi_{\ell},A)$ in the upper bounds of \eqref{eq2.16} and \eqref{eq2.17}, are estimated by using Strategy I and Strategy II, respectively.
  }
 \end{flushleft}
 \end{table}
}

\section{Concluding remarks}

In this paper we consider the computations, error analysis, implementations and applications
of $\varphi$-functions for large sparse matrices with low rank or with fast decaying singular values. Given a
sparse column-row approximation of $A$, we take into account how to compute the matrix function series $\varphi_{\ell}(A)~(\ell=0,1,2,\ldots,p)$ efficiently,
and to estimate their 2-norm Fr\'{e}chet relative and absolute condition numbers effectively.

The numerical behavior of our new method is closely related to that of reduced-rank approximation of large sparse matrices \cite{Stewart,GWStewart}. Thus, a promising research area is to seek new technologies to improve the performance of the sparse column-row approximation algorithm on very large matrices. Another interesting topic is to combine other advanced algorithms such as the randomized singular value
decomposition algorithm \cite{Random,Mon} with our new strategies for the computation of functions of large sparse matrices.

\section*{Acknowledgments}
We would like to thank Juan-juan Tian for helpful discussions.

\end{document}